\documentclass[12pt,reqno]{amsart}
\usepackage[utf8]{inputenc}
\usepackage[T1]{fontenc}
\usepackage{amsmath}
\usepackage{amsfonts}
\usepackage{amssymb,amsthm}
\usepackage{graphicx}
\usepackage{mathrsfs}
\usepackage{hyperref}
\usepackage{color}
\usepackage{etoolbox}
\usepackage{comment}
\usepackage{fullpage}
\usepackage{tikz}
\newtheorem{thm}{Theorem}[section]
\newtheorem{defi}[thm]{Definition}
\newtheorem{pro}[thm]{Proposition}
\newtheorem{cor}[thm]{Corollary}
\newtheorem{lem}[thm]{Lemma}
\newtheorem{rem}[thm]{Remark}

\newcommand{\R}{\mathbb{R}}
\newcommand{\N}{\mathbb{N}}

\makeatletter
\@addtoreset{equation}{section}
\makeatother

\title{Existence Result for Difference Equations on Non-Uniform Grids via Upper and Lower Solution Method}

\author{Shalmali Bandyopadhyay, Kimsear Lor}

\date{}

\begin{document}

\begin{abstract}
This paper establishes an existence theory for discrete second-order boundary value problems on non-uniform time grids using the upper and lower solution method. We consider difference equations of the form $u^{\Delta\Delta}(t_{i-1}) + f(t_i, u(t_i), u^\Delta(t_{i-1})) = 0$ on a non-uniform time grid ${t_0, t_1, \ldots, t_{n+2}}$ with mixed boundary conditions $u^\Delta(t_0) = 0$ and $u(t_{n+2}) = g(t_{n+2})$. This extends previous work on homogeneous boundary conditions to the non-homogeneous case, requiring a sophisticated functional analytic framework to handle the resulting affine function spaces. Our approach employs a decomposition strategy that separates boundary effects from the differential structure, enabling the application of Brouwer's Fixed Point Theorem to establish existence with solutions bounded between upper and lower functions.
\end{abstract}
\maketitle



\section{Introduction}

Difference equations arise naturally in numerous scientific and engineering applications where discrete time modeling is essential or preferred over continuous formulations. In computational finance discrete-time models are fundamental for option pricing, where finite difference methods are used (e.g. see  \cite{Schwartz1977}) to price options by approximating continuous-time differential equations with discrete-time difference equations, building on foundational work from the 1970s. Economic modeling extensively employs \cite{FergusonLim2003} discrete time formulations, as economic behavior is inherently dynamic and much of economic analysis is based on discrete time periods such as months, quarters, or years, reflecting the periodic nature of data collection and decision-making in financial markets and policy analysis.

Simultaneously, population dynamics modeling represents a crucial domain where difference equations have a long history of use as discrete time models (see  \cite{Cushing2019}), describing typically autonomous discrete time dynamics. Mathematical biology applications continue to expand, where discrete Lotka-Volterra models are more appropriate than continuous ones when populations are of non-overlapping generations, with studies showing that discrete time models governed by difference equations provide more realistic descriptions of population interactions (\cite{Din2013}). 
Contemporary epidemiological applications have demonstrated the importance of \cite{Din2021,DelaMaza2021} discrete SEIR models, where discrete forms are discussed to describe viral disease propagation and find control mechanisms for rapid spread of infectious diseases like COVID-19. 

Engineering applications particularly benefit from digital signal processing, where difference equations characterize digital filters and provide the mathematical foundation for implementing finite impulse response and infinite impulse response filters in real-time systems (e.g. \cite{Shenoi2005,Baraniuk2020} ). Additionally, control theory applications utilize difference equations (\cite{Akman2016}) for discrete-time system analysis, where linear constant-coefficient difference equations model linear time-invariant systems and enable the design of digital controllers for various engineering applications. Classical applications in \cite{VeldmanRinzema1992} non-uniform grid discretizations continue to show that careful consideration of truncation errors introduced by variable step sizes is essential for accurate numerical solutions in computational fluid dynamics and partial differential equation solvers. The theoretical foundations are well established in \cite{Elaydi2008,Sayama2024} comprehensive treatments of discrete chaos and complex systems modeling.

In this manuscript we establish existence result of a difference equation on nonuniform grid with mixed boundary condition by employing standard upper and lower solution method, extensively available in the literature together with Brouwer fixed point theory. Prior to stating our problem, we begin by establishing the mathematical framework for non-uniform time grids and associated difference calculus. For further details regarding difference calculus, we refer readers to the book \cite{KelleyPeterson1991}.

\begin{defi}\label{def:grid}
Let $n \in \N$ be fixed and $a < b$ be real numbers. A \emph{non-uniform time grid} is a strictly increasing sequence $\mathbb{T} = \{t_0, t_1, \ldots, t_{n+2}\}$ where $a = t_0 < t_1 < \ldots < t_{n+2} = b$.
\end{defi}
\begin{defi}\label{def:stepsize}
Define the step sizes $h_i = t_{i+1} - t_i$ for $i = 0, 1, \ldots, n+1$, and assume there exist constants $h_{\min}, h_{\max} > 0$ such that $0 < h_{\min} \leq h_i \leq h_{\max}$ for all $i$.
\end{defi}

\begin{defi}\label{def:delta_derivative}
For a function $u: \mathbb{T} \to \R$, the \emph{delta derivative} at $t_i$ is defined by
\begin{equation*}
u^\Delta(t_i) = \frac{u(t_{i+1}) - u(t_i)}{t_{i+1} - t_i} = \frac{u(t_{i+1}) - u(t_i)}{h_i}
\end{equation*}
for $i = 0, 1, \ldots, n+1$.
\end{defi}

\begin{defi}\label{def:second_derivative}
The \emph{second delta derivative} is defined as
\begin{equation*}
u^{\Delta\Delta}(t_i) = \frac{u^\Delta(t_{i+1}) - u^\Delta(t_i)}{t_{i+1} - t_i} = \frac{u^\Delta(t_{i+1}) - u^\Delta(t_i)}{h_i}
\end{equation*}
for $i = 0, 1, \ldots, n$.
\end{defi}
Having introduced necessary definitions, we now consider the following boundary value problem.
\begin{equation}\label{eq:main_problem}
\begin{cases}
u^{\Delta\Delta}(t_{i-1}) + {f}(t_i, u(t_i), u^\Delta(t_{i-1})) = 0, & i \in \{1, 2, \ldots, n+1\} \\
u^\Delta(t_0) = 0 \\
u(t_{n+2}) = g(t_{n+2})
\end{cases}
\end{equation}
where $\{t_0, t_1, \ldots, t_{n+2}\}$ is a non-uniform partition of an interval $[a,b]$ with $a = t_0 < t_1 < \ldots < t_{n+2} = b$. Moreover, we assume

\noindent \textbf{(A1)} $f(t_i, \cdot, \cdot)$ is continuous on $\{t_1, \ldots, t_{n+1}\}$, right continuous at $t_0$ and left continuous at $t_{n+2}$ in the discrete sense; $f(\cdot, x,y)$ is continuous on $\mathbb{R}^2$.

\noindent \textbf{(A2)} ${f}$ is non-increasing in its third variable

\noindent \textbf{(A3)} $g: \mathbb{T} \to \R^+$ is continuous (in the discrete sense)
In \cite{kunkel2019}, authors studied the following problem, 
\begin{equation*}
\begin{cases}
u^{\Delta\Delta}(t_{i-1}) + {f}(t_i, u(t_i), u^\Delta(t_{i-1})) = 0, & i \in \{1, 2, \ldots, n+1\} \\
u^\Delta(t_0) = 0 \\
u(t_{n+2}) = 0.
\end{cases}
\end{equation*}
Later, in \cite{KunkelMartin}, authors extended the above problem for $n^{\text{th}}$ order derivative. In our case, the second boundary condition allows for non-homogeneous boundary values. The main challenge we faced was defining the underlying function space to accommodate this non-homogeneous structure. This problem requires a more sophisticated functional analytic approach compared to the standard homogeneous case. While homogeneous boundary conditions naturally define linear function spaces that are well-suited for fixed point theory applications, non-homogeneous conditions create affine spaces that lack the linear structure necessary for classical fixed point arguments. To overcome this challenge, we employ a careful decomposition strategy that separates the boundary effects from the core differential structure, allowing the problem to be reformulated within an appropriate linear framework while preserving the original boundary requirements (see Section \ref{sec:prelim}).

Prior to stating our theorem, we will provide definitions of solution, upper solution and lower solution.
\begin{defi}\label{def:solution}
A function $u: \mathbb{T} \to \R$ is a \emph{solution} of \eqref{eq:main_problem} if it satisfies both the difference equation and the boundary conditions.
\end{defi}

\begin{defi}\label{def:lower_solution}
A function $\alpha: \mathbb{T} \to \R$ is called a \emph{lower solution} of \eqref{eq:main_problem} if:
\begin{equation}
\begin{cases}
\alpha^{\Delta\Delta}(t_{i-1}) + {f}(t_i, \alpha(t_i), \alpha^\Delta(t_{i-1})) \geq 0, & i \in \{1, \ldots, n+1\} \\
\alpha^\Delta(t_0) \geq 0 \\
\alpha(t_{n+2}) \leq g(t_{n+2})
\end{cases}
\end{equation}
\end{defi}

\begin{defi}\label{def:upper_solution}
A function $\beta: \mathbb{T} \to \R$ is called an \emph{upper solution} of \eqref{eq:main_problem} if:
\begin{equation}
\begin{cases}
\beta^{\Delta\Delta}(t_{i-1}) + {f}(t_i, \beta(t_i), \beta^\Delta(t_{i-1})) \leq 0, & i \in \{1, \ldots, n+1\} \\
\beta^\Delta(t_0) \leq 0 \\
\beta(t_{n+2}) \geq g(t_{n+2})
\end{cases}
\end{equation}
\end{defi}

\begin{thm}\label{thm:main}
Let $\alpha$ and $\beta$ be lower and upper solutions of \eqref{eq:main_problem}, respectively, with $\alpha(t_i) \leq \beta(t_i)$ for all $i \in \{0, 1, \ldots, n+1\}$. Assume conditions \textbf{(A1)}-\textbf{(A3)} hold. Then \eqref{eq:main_problem} has a solution $u$ satisfying
\begin{equation}
\alpha(t_i) \leq u(t_i) \leq \beta(t_i), \quad i \in \{0, 1, \ldots, n+2\}
\end{equation}
\end{thm}
The remainder of this paper is organized as follows. Section \ref{sec:prelim} establishes the functional analytic framework, including the definition of appropriate function spaces, the solution decomposition theorem that separates boundary conditions from the difference equation, and Brouwer's Fixed Point Theorem in the discrete setting. Section \ref{sec:existence} presents the main existence result through three key steps: construction of an auxiliary problem with modified nonlinearity that handles solutions outside the upper and lower bounds, development of a solution operator whose fixed points correspond to solutions of the auxiliary problem, and application of Brouwer's theorem to establish existence. Finally, Section \ref{sec:bound} completes the proof by verifying that any solution of auxiliary problem remains within the prescribed upper and lower bounds through a contradiction argument utilizing the maximum principle for discrete second derivatives to guarantee existence of solution of original problem.
\section{Preliminaries}
\label{sec:prelim}
We begin by defining the underlying function spaces that are required for our purposes. 
\begin{defi}
    The space $E$ is defined as $$E = \{u: \mathbb{T} \to \R : u^\Delta(t_0) = 0, u(t_{n+2}) = g(t_{n+2})\}$$
\end{defi} 

\begin{defi}
The space $E_0$ is defined as
\begin{equation*}
E_0 = \{v: \mathbb{T} \to \mathbb{R} : v^\Delta(t_0) = 0, v(t_{n+2}) = 0\}.
\end{equation*}
\end{defi}

\begin{pro}
The space $E_0$ equipped with the supremum norm $\|v\|_\infty = \max_{t \in \mathbb{T}} |v(t)|$ is a normed linear space.
\end{pro}

\begin{proof}
We verify that $E_0$ is a linear space and that the supremum norm satisfies the norm axioms.

\textbf{Linear space properties:}

\textit{Closure under addition:} Let $v_1, v_2 \in E_0$. Then
\begin{align*}
(v_1 + v_2)^\Delta(t_0) &= v_1^\Delta(t_0) + v_2^\Delta(t_0) = 0 + 0 = 0,\\
(v_1 + v_2)(t_{n+2}) &= v_1(t_{n+2}) + v_2(t_{n+2}) = 0 + 0 = 0.
\end{align*}
Therefore, $v_1 + v_2 \in E_0$.

\textit{Closure under scalar multiplication:} Let $v \in E_0$ and $\alpha \in \mathbb{R}$. Then
\begin{align*}
(\alpha v)^\Delta(t_0) &= \alpha v^\Delta(t_0) = \alpha \cdot 0 = 0,\\
(\alpha v)(t_{n+2}) &= \alpha v(t_{n+2}) = \alpha \cdot 0 = 0.
\end{align*}
Therefore, $\alpha v \in E_0$.

\textit{Zero element:} The zero function $0: \mathbb{T} \to \mathbb{R}$ defined by $0(t) = 0$ for all $t \in \mathbb{T}$ satisfies $0^\Delta(t_0) = 0$ and $0(t_{n+2}) = 0$, so $0 \in E_0$.

\textit{Additive inverse:} For any $v \in E_0$, we have $-v \in E_0$ since $(-v)^\Delta(t_0) = -v^\Delta(t_0) = 0$ and $(-v)(t_{n+2}) = -v(t_{n+2}) = 0$.

\textbf{Norm properties:}

\textit{Positive definiteness:} $\|v\|_\infty \geq 0$ for all $v \in E_0$, and $\|v\|_\infty = 0$ if and only if $|v(t)| = 0$ for all $t \in \mathbb{T}$, which occurs if and only if $v$ is the zero function.

\textit{Homogeneity:} For any $\alpha \in \mathbb{R}$ and $v \in E_0$,
\begin{equation*}
\|\alpha v\|_\infty = \max_{t \in \mathbb{T}} |\alpha v(t)| = |\alpha| \max_{t \in \mathbb{T}} |v(t)| = |\alpha| \|v\|_\infty.
\end{equation*}

\textit{Triangle inequality:} For any $v_1, v_2 \in E_0$,
\begin{equation*}
\|v_1 + v_2\|_\infty = \max_{t \in \mathbb{T}} |v_1(t) + v_2(t)| \leq \max_{t \in \mathbb{T}} (|v_1(t)| + |v_2(t)|) \leq \|v_1\|_\infty + \|v_2\|_\infty.
\end{equation*}
\end{proof}

\begin{defi}
Define the particular solution $w_p: \mathbb{T} \to \mathbb{R}$ by
\begin{equation}
w_p(t_k) = g(t_{n+2}) \quad \text{for all } k =0, 1, \ldots, n+2.
\end{equation}
\end{defi}

\begin{lem}
The function $w_p$ satisfies the boundary conditions of problem \eqref{eq:main_problem}, i.e., $w_p^\Delta(t_0) = 0$ and $w_p(t_{n+2}) = g(t_{n+2})$.
\end{lem}

\begin{proof}
Since $w_p$ is constant on $\mathbb{T}$, we have
\begin{equation*}
w_p^\Delta(t_0) = \frac{w_p(t_1) - w_p(t_0)}{h_0} = \frac{g(t_{n+2}) - g(t_{n+2})}{h_0} = 0,
\end{equation*}
and clearly $w_p(t_{n+2}) = g(t_{n+2})$.
\end{proof}

\begin{thm}[Solution Decomposition]
Every function $u \in E$ can be uniquely written as
\begin{equation*}
u = v + w_p,
\end{equation*}
where $v \in E_0$ and $w_p$ is the particular solution defined above.
\end{thm}

\begin{proof}
Given $u \in E$, define $v = u - w_p$. Then
\begin{align*}
v^\Delta(t_0) &= u^\Delta(t_0) - w_p^\Delta(t_0) = 0 - 0 = 0,\\
v(t_{n+2}) &= u(t_{n+2}) - w_p(t_{n+2}) = g(t_{n+2}) - g(t_{n+2}) = 0.
\end{align*}
Therefore, $v \in E_0$. Uniqueness follows from the fact that if $u = v_1 + w_p = v_2 + w_p$, then $v_1 = v_2$.
\end{proof}
\begin{cor}
    If $v \in E_0$ satisfies the DE in Eq. \eqref{eq:main_problem}, then $u \in E$ is a solution of \eqref{eq:main_problem}.
\end{cor}
\begin{proof}
    The proof immediately follows from the fact that $v^{\Delta\Delta}=u^{\Delta\Delta}$.
\end{proof}

\begin{rem}{\rm
The decomposition $u = v + w_p$ allows us to analyze the operator $T: E \to E$ defined in equation \eqref{eq:T} by exploiting the linear structure of $E_0$ while maintaining the affine boundary conditions of the original problem. This framework provides the necessary functional analytic foundation for applying Brouwer's Fixed Point Theorem in Section 3.3.}
\end{rem}

\begin{thm}[Brouwer's Fixed Point Theorem]
Let $K$ be a nonempty, compact, convex subset of a finite-dimensional normed space $X$, and let $T: K \to K$ be a continuous mapping. Then $T$ has at least one fixed point in $K$.
\end{thm}
Let us first understand this concept of fixed point theory for a single-variable function and how and when we can guarantee that a fixed point exists.

\noindent \textbf{Claim}: Let $f: [0,1] \rightarrow [0,1]$, and it is continuous in that domain; then a fixed point of the function $f$ exists in the same domain.

\begin{proof}
    Let $g(x) = f(x) - x$, where $g(x)$ is continuous, since $f(x)$ is continuous and $x$ is an identity function.  Notice that within our domain, $f(0)>0$ and $f(1)<1$.
    
Therefore,
    \begin{align*}
        g(0) = f(0) - 0 = f(0) > 0 \Rightarrow g(0) > 0\\
        g(1) = f(1) - 1 = f(1) - 1 < 0 \Rightarrow g(1) < 0
    \end{align*}
    
Knowing that $g(0)>0$ and $g(1)<0$, we can apply the Intermediate Value Theorem and conclude that there exists a value $c$ in the established domain such that $g(c)=0$ $\Rightarrow f(c)-c=0$, therefore $f(c) = c$, i.e. $c$ is a fixed point of $f(x)$. 
\end{proof}
\begin{rem}{\rm
 Note that the range of our interval does not allow $f(x)$ to go beyond 1 or below 0; in other words, $f(1) \not> 1$ and $f(0) \not< 0$. And if $f(0)=0$ or $f(1)=1$, we would already have a fixed point of $f(x)$ either $0$ or $1$. }   
\end{rem}

\begin{cor}[Application to Discrete Boundary Value Problems]
\label{cor:Brouwer}
Let $\overline{B_r} = \{u \in E : \|u\|_\infty \leq r\}$ for some $r > 0$. If $T: E \to E$ is continuous and satisfies $T(\overline{B_r}) \subseteq \overline{B_r}$, then $T$ has a fixed point in $\overline{B_r}$.
\end{cor}

\begin{proof}
Since $\mathbb{T} = \{t_0, t_1, \ldots, t_{n+2}\}$ is finite with $|\mathbb{T}| = n+3$ elements, the space $E$ can be identified with a subset of $\mathbb{R}^{n+3}$ via the isomorphism
\begin{equation*}
\Phi: E \to \mathbb{R}^{n+3}, \quad \Phi(u) = (u(t_0), u(t_1), \ldots, u(t_{n+2})).
\end{equation*}
However, due to the boundary conditions $u^\Delta(t_0) = 0$ and $u(t_{n+2}) = g(t_{n+2})$, the set $E$ is actually isomorphic to $\mathbb{R}^{n+1}$. Indeed, using the decomposition $u = v + w_p$ where $v \in E_0$, the space $E_0$ can be parameterized by the values $(v(t_1), v(t_2), \ldots, v(t_{n+1}))$ since:
\begin{itemize}
\item[] $v(t_{n+2}) = 0$ (fixed by boundary condition)
\item[] The constraint $v^\Delta(t_0) = 0$ gives us $v(t_1) = v(t_0)$
\end{itemize}

Therefore, $E_0$ is isomorphic to $\mathbb{R}^{n+1}$, and consequently $E$ is also $(n+1)$-dimensional.

The set ${\overline{B_r}} = \{u \in E : \|u\|_\infty \leq r\}$ is:
\begin{itemize}
\item[] \textbf{Bounded:} By definition, $\|u\|_\infty \leq r$ for all $u \in \overline{B_r}$.
\item[] \textbf{Closed:} If $u_k \in \overline{B_r}$ and $u_k \to u$ in $E$, then $\|u_k\|_\infty \leq r$ for all $k$. By continuity of the norm, $\|u\|_\infty = \lim_{k \to \infty} \|u_k\|_\infty \leq r$, so $u \in \overline{B_r}$.
\item[] \textbf{Convex:} For any $u_1, u_2 \in \overline{B_r}$ and $\lambda \in [0,1]$, we need to show $\lambda u_1 + (1-\lambda) u_2 \in E$ and $\|\lambda u_1 + (1-\lambda) u_2\|_\infty \leq r$.
\end{itemize}

For convexity, let $u_1, u_2 \in \overline{B_r}$ and $\lambda \in [0,1]$. Set $u_\lambda = \lambda u_1 + (1-\lambda) u_2$. Using the decomposition $u_i = v_i + w_p$:
\begin{equation*}
u_\lambda = \lambda(v_1 + w_p) + (1-\lambda)(v_2 + w_p) = \lambda v_1 + (1-\lambda) v_2 + w_p.
\end{equation*}
Since $E_0$ is a linear space, $\lambda v_1 + (1-\lambda) v_2 \in E_0$, so $u_\lambda \in E$. Moreover,
\begin{equation}
\|u_\lambda\|_\infty = \|\lambda u_1 + (1-\lambda) u_2\|_\infty \leq \lambda \|u_1\|_\infty + (1-\lambda) \|u_2\|_\infty \leq \lambda r + (1-\lambda) r = r.
\end{equation}

Since $E$ is finite-dimensional and $\overline{B_r}$ is closed and bounded, $\overline{B_r}$ is compact. The hypotheses of Brouwer's theorem are satisfied, guaranteeing the existence of a fixed point.
\end{proof}

\section{Existence Result}
\label{sec:existence}
In this section, we prove the existence result stated in Theorem \eqref{thm:main}. The proof follows three main steps: first, we construct an auxiliary problem with modified nonlinearity to handle cases where the original nonlinearity $f$ lies outside the bounds defined by our upper and lower solutions; second, we define a solution operator $T$ whose fixed points correspond to solutions of Problem \eqref{eq:main_problem}; and third, we apply Brouwer's Fixed Point Theorem to establish existence of a fixed point of the operator $T$.
\subsection{Construction of Auxiliary Problem}

Observe that, the nonlinearity $f$ need not necessarily be in between the upper and lower solution. Therefore, we define the following auxiliary function $\widetilde{f}$, for $i=1,2, \ldots, n+1$
\begin{equation}
\label{eq:tilde-f}
\widetilde{f}\left(t_i, x, \frac{x - z} {h_{i-1}} \right) = \begin{cases}
f\left(t_i, \beta(t_i), \frac{\beta(t_i)-\sigma(t_{i}, z)}{h_{i-1}}\right) - \frac{x-\beta(t_{i})}{x-\beta(t_{i}) + 1}, & \text{if } x > \beta(t_{i}), \\[0.5em]
f\left(t_i, x, \frac{x - \sigma(t_{i}, z)}{h_{i-1}}\right), & \text{if } \alpha(t_{i}) \leq x \leq \beta(t_{i}), \\[0.5em]
f\left(t_i, \alpha(t_i), \frac{\alpha(t_i)-\sigma(t_{i}, z)}{h_{i-1}}, \right)  + \frac{\alpha(t_{i})-x}{ \alpha(t_{i})-x + 1}, & \text{if } x < \alpha(t_i),
\end{cases}
\end{equation}
where,
\begin{equation}
\label{eq:sigma}
\sigma(t_{i}, z) = \begin{cases}
 \beta(t_{i-1}), & \text{if } z > \beta(t_{i-1}),\\
z, & \text{if } \alpha(t_{i-1}) \leq z \leq \beta(t_{i-1}), \\
\alpha(t_{i-1}), & \text{if } z < \alpha(t_{i-1}).
\end{cases}
\end{equation}
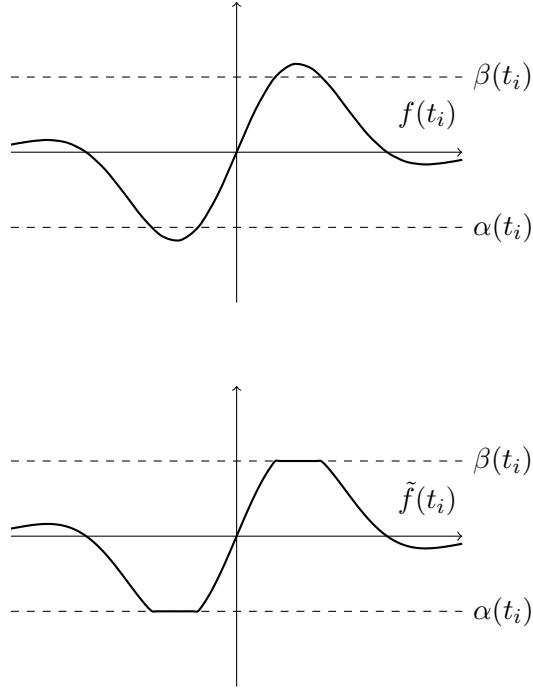
\begin{figure}[h]
\setlength{\unitlength}{.6cm}
\begin{center}
\begin{tikzpicture}
    \draw[->] (-3,0) -- (3,0); 
    \draw[->] (0,-2) -- (0,2); 

    \draw[thick, domain=-3:3, smooth] plot (\x, {1.5*exp(-0.3*\x*\x)*sin(180*\x/2)});
    
    \draw[dashed] (-3,1) -- (3,1) node[right] {\small $\beta(t_i)$};
    \draw[dashed] (-3,-1) -- (3,-1) node[right] {\small $\alpha(t_i)$};

    \node[right] at (2,.5) {\small $f(t_i)$};

\end{tikzpicture}
\end{center}
\end{figure}
\begin{figure}[h]
\setlength{\unitlength}{.6cm}
\begin{center}
\begin{tikzpicture}
    \draw[->] (-3,0) -- (3,0); 
    \draw[->] (0,-2) -- (0,2); 

    \draw[dashed] (-3,1) -- (3,1) node[right] {\small $\beta(t_i)$};
    \draw[dashed] (-3,-1) -- (3,-1) node[right] {\small $\alpha(t_i)$};

    \draw[thick, domain=-3:3, smooth, samples=100] 
        plot (\x, {max(-1, min(1.5*exp(-0.3*\x*\x)*sin(180*\x/2), 1))});
    
    \node[right] at (2,0.5) {\small $\Tilde{f}(t_i)$};

\end{tikzpicture}
\end{center}
\end{figure}
\begin{pro}
The auxiliary function $\widetilde{f}$ is continuous and bounded. Moreover, there exists a constant $M > 0$ such that $|\widetilde{f}| \leq M$.
\end{pro}

\begin{proof}
We prove continuity and boundedness separately.

\textbf{Part 1: Continuity}

To prove $\widetilde{f}$ is continuous, we need to:
\begin{enumerate}
\item Show continuity within each piece
\item Show continuity at each breakpoint: $x = \alpha(t_i)$ and $x = \beta(t_i)$
\end{enumerate}

\textit{Step 1: Continuity within each piece}

When $x > \beta(t_i)$
$$\widetilde{f} = f\left(t_i, \beta(t_i), \frac{\beta(t_i) - \sigma(t_i, z)}{h_{i-1}}\right) - \frac{x - \beta(t_i)}{x - \beta(t_i) + 1}$$
Since $f$ is continuous by assumption {\bf(A1)}, $\sigma(t_i, z)$ is piecewise constant in $z$, and $x > \beta(t_i)$ implies $x - \beta(t_i) + 1 > 1 \neq 0$ implies this piece is continuous.

When $\alpha(t_i) \leq x \leq \beta(t_i)$
$$\widetilde{f} = f\left(t_i, x, \frac{x - \sigma(t_i, z)}{h_{i-1}}\right)$$
This is continuous by assumption {\bf(A1)} since $f$ is continuous.

When $x < \alpha(t_i)$
$$\widetilde{f} = f\left(t_i, \alpha(t_i), \frac{\alpha(t_i) - \sigma(t_i, z)}{h_{i-1}}\right) + \frac{\alpha(t_i) - x}{\alpha(t_i) - x + 1}$$
Since $x < \alpha(t_i)$ implies $\alpha(t_i) - x + 1 > 1 \neq 0$, and $f$ is continuous implies this piece is continuous.

\textit{Step 2: Continuity at $x = \beta(t_i)$ and at $x = \alpha(t_i)$}

We need to show:
$$\lim_{x \to \beta(t_i)^-} \widetilde{f} = \lim_{x \to \beta(t_i)^+} \widetilde{f} = \widetilde{f}(\beta(t_i))$$

\noindent {Left-hand limit:}
$$\lim_{x \to \beta(t_i)^-} \widetilde{f} = f\left(t_i, \beta(t_i), \frac{\beta(t_i) - \sigma(t_i, z)}{h_{i-1}}\right)$$

\noindent {Right-hand limit:} As $x \to \beta(t_i)^+$, we have $x - \beta(t_i) \to 0^+$, so $\frac{x - \beta(t_i)}{x - \beta(t_i) + 1} \to 0$. Therefore:
$$\lim_{x \to \beta(t_i)^+} \widetilde{f} = f\left(t_i, \beta(t_i), \frac{\beta(t_i) - \sigma(t_i, z)}{h_{i-1}}\right)$$

\noindent {Function value at the break point:}
$$\widetilde{f}(\beta(t_i)) = f\left(t_i, \beta(t_i), \frac{\beta(t_i) - \sigma(t_i, z)}{h_{i-1}}\right)$$

All three values are equal, so $\widetilde{f}$ is continuous at $x = \beta(t_i)$. Similarly, $\widetilde{f}$ is continuous at $x = \alpha(t_i)$.

\textbf{Part 2: Boundedness}
Let $\mathcal{D}$ be the domain of $\widetilde{f}$. Let us understand this domain carefully. Since $\widetilde{f}$ is a function of three variables, and each variable is a real number, clearly $\mathcal{D} \subset \mathbb{R}^3$. The first argument of $\widetilde{f}$ is $t_i$, where $t_i \in \mathbb{T}=\{t_0,t_1, \ldots, t_{n+2}\}$. $\mathbb{T}$ is a set of isolated points from $\mathbb{R}$, hence closed and bounded. Next comes the second and third argument of $\widetilde{f}$, which is $u(t_i)$. Note that if $u$ is a solution of Problem \eqref{eq:aux_problem}, then $u$ must be twice differentiable function in discrete sense, in other words, both $u(t_i)$ and $u^{\Delta}(t_i)$ is continuous on $\mathbb{T}$. Therefore, by extreme value theorem $u$ and $u^{\Delta}$ both reaches their maximum and minimum on the closed set $\mathbb{T}$. Consequently, $\mathcal{D} \subset K \subset \mathbb{R}^3$, where $K$ is compact. Finally, continuity of $\widetilde{f}$ implies its boundedness on $K$, hence on $\mathcal{D}$. In particular, there exists a constant $M > 0$ such that $|\widetilde{f}| \leq M$ for any $t_i \in \mathbb{T}$.
\end{proof}

Therefore, we study the auxiliary problem below:
\begin{equation}\label{eq:aux_problem}
\begin{cases}
u^{\Delta\Delta}(t_{i-1}) + \widetilde{f}(t_i, u(t_i), u^\Delta(t_{i-1})) = 0, & i \in \{1, 2, \ldots, n+1\} \\
u^\Delta(t_0) = 0 \\
u(t_{n+2}) = g(t_{n+2})
\end{cases}
\end{equation}
\noindent \textbf{Claim:} If $\alpha(t_i) \leq u(t_i) \leq \beta(t_i)$ for all $i \in \{ 0, 1, \ldots, n+2\}$, then Auxiliary Problem \eqref{eq:aux_problem} is equivalent to Main Problem \eqref{eq:main_problem}.
\begin{proof}
Note that, when we plug in $z = u(t_{i-1})$ in the expression of $\sigma(t_i, z)$ given in \eqref{eq:sigma}, we get
$$\sigma(t_i, u(t_{i-1})) = \begin{cases}
\beta(t_{i-1}) & \text{if } u(t_{i-1}) > \beta(t_{i-1}) \\
u(t_{i-1}) & \text{if } \alpha(t_{i-1}) \leq u(t_{i-1}) \leq \beta(t_{i-1}) \\
\alpha(t_{i-1}) & \text{if } u(t_{i-1}) < \alpha(t_{i-1})
\end{cases}$$

which implies,
$$\sigma(t_i, u(t_{i-1})) = u(t_{i-1}) \quad \text{if } \alpha(t_{i-1}) \leq u(t_{i-1}) \leq \beta(t_{i-1}).$$

Therefore, if $\alpha(t_i) \leq u(t_i) \leq \beta(t_i)$ for all $i \in \{0, 1, \ldots, n+2\}$, then we get
\begin{align*}
\widetilde{f}(t_i, u(t_i), u^{\Delta}(t_{i-1})) &= \widetilde{f}\left(t_i, u(t_i), \frac{u(t_{i}) - u(t_{i-1})}{t_{i} - t_{i-1}}\right) \\
&= {f}\left(t_i, u(t_i), \frac{u(t_i) - \sigma(t_i, u(t_{i-1}))}{t_i - t_{i-1}}\right) \\
&= f\left(t_i, u(t_i), \frac{u(t_i) - u(t_{i-1})}{t_i - t_{i-1}}\right)\\
&={f}(t_i, u(t_i), u^{\Delta}(t_{i-1})),
\end{align*}
which establishes the claim.
\end{proof}

Consequently, if we can find a solution $u(t_i)$ of Problem \eqref{eq:aux_problem} such that $\alpha(t_i) \leq u(t_i) \leq \beta(t_i)$ for all $i \in \{0, 1, \ldots, n+2\}$, then we can conclude that the solution $u(t_i)$ is in fact a solution of Problem \eqref{eq:main_problem}

\subsection{Construction of the Operator $T$}
With the auxiliary problem established, we now construct the solution operator $T$ whose fixed points are in fact  solutions of Problem \eqref{eq:main_problem}. In order to do that, we establish the following two lemmas which are essential for our purposes.

\begin{lem}
\label{lem:delta_formula}
If $u$ satisfies \eqref{eq:main_problem}, then
\begin{equation*}
u^\Delta(t_k) = -\sum_{i=1}^{k} h_{i-1} \widetilde{f}(t_i, u(t_i), u^\Delta(t_{i-1}))
\end{equation*}
for $k = 1, 2, \ldots, n+1$.
\end{lem}

\begin{proof}
Here we utilize proof by induction. First, let us show the base cases. We start with $k = 1$:
\begin{align}
u^{\Delta\Delta}(t_0) &= \frac{u^\Delta(t_1) - u^\Delta(t_0)}{h_0} \notag\\
 \Rightarrow  -\widetilde{f}(t_1, u(t_1), u^\Delta(t_0)) &= \frac{u^\Delta(t_1) - 0}{h_0}\notag \\
 \Rightarrow u^\Delta(t_1) &= -h_0 \widetilde{f}(t_1, u(t_1), u^\Delta(t_0)) \notag
\end{align}
Next, let us check for ${u^\Delta(t_2)} \colon$
\begin{align}
    u^{\Delta\Delta}(t_1) &= \frac{u^\Delta(t_2)-u^\Delta(t_1)}{h_1} \notag \\
 \Rightarrow   -h_1\widetilde{f}(t_2, u(t_2), u^\Delta(t_1))&=u^\Delta(t_2)-u^\Delta(t_1)\notag \\
  \Rightarrow  u^\Delta(t_2)&=-h_1\widetilde{f}(t_2, u(t_2), u^\Delta(t_1))-h_0\widetilde{f}(t_1, u(t_1), u^\Delta(t_0)) \notag\\
  & = -\displaystyle \sum_{i=1}^{2} h_{i-1}\widetilde{f}(t_i, u(t_i), u^\Delta(t_{i-1})) \notag
\end{align}
Then, let us check for the the case $k=3\colon$
\begin{align*}
    u^{\Delta\Delta}(t_2) &= \frac{u^\Delta(t_3)-u^\Delta(t_2)}{h_2} \\
   \Rightarrow -h_2\widetilde{f}(t_3, u(t_3), u^\Delta(t_2))&=u^\Delta(t_3)-u^\Delta(t_2) \\
    \Rightarrow u^\Delta(t_3)&=-h_2\widetilde{f}(t_3, u(t_3), u^\Delta(t_2))-h_1\widetilde{f}(t_2, u(t_2), u^\Delta(t_1))-\\
    &h_0\widetilde{f}(t_1, u(t_1), u^\Delta(t_0))\\
    &= - \displaystyle \sum_{i=1}^3 h_{i-1}\widetilde{f}(t_i, u(t_i), u^{\Delta}(t_{i-1}))
\end{align*}
Checking the pattern, let us assume the formula holds for $k$, and 
\begin{equation*}
    u^\Delta(t_k)=-\displaystyle \sum_{i=1}^{k}{h_{i-1}\widetilde{f}(t_i, u(t_i), u^\Delta(t_{i-1}))}
\end{equation*}
Thereafter, we will show this formula holds for $t_{k+1}$.
\begin{align*}
u^{\Delta\Delta}(t_k) &= \frac{u^\Delta(t_{k+1}) - u^\Delta(t_k)}{h_k}\\
\Rightarrow -\widetilde{f}(t_{k+1}, u(t_{k+1}), u^\Delta(t_k)) & =\frac{u^\Delta(t_{k+1}) - u^\Delta(t_k)}{h_k}\\
u^\Delta(t_{k+1}) &= u^\Delta(t_k) - h_k\widetilde{f}(t_{k+1}, u(t_{k+1}), u^\Delta(t_k))\\
& = -\displaystyle \sum_{i=1}^{k}{h_{i-1}\widetilde{f}(t_i, u(t_i), u^\Delta(t_{i-1}))}- h_k\widetilde{f}(t_{k+1}, u(t_{k+1}), u^\Delta(t_k))\\
& = -\displaystyle \sum_{i=1}^{k+1}{h_{i-1}\widetilde{f}(t_i, u(t_i), u^\Delta(t_{i-1}))}
\end{align*}
From here, we can conclude, by mathematical induction, 
\begin{equation}
\label{eq:uDelta}
u^\Delta(t_k) = -\sum_{i=1}^{k} h_{i-1} \widetilde{f}(t_i, u(t_i), u^\Delta(t_{i-1}))
\end{equation}
for $k = 1, 2, \ldots, n+1$. This completes the proof.
\end{proof}

\begin{lem}\label{lem:solution_formula}
If $u$ satisfies \eqref{eq:main_problem}, then 
\begin{equation}
\label{eq:u-k}
    u(t_{k})=g(t_{n+2})+\displaystyle\sum_{j=k}^{n+1}{h_j\displaystyle\sum_{i=1}^{j}{h_{i-1}\widetilde{f}(t_i, u(t_i), u^\Delta(t_{i-1}))}}
\end{equation}
\end{lem}

\begin{proof}
To prove this lemma, we employ reverse mathematical induction. We first establish the base cases for $k = n+2$, $k = n+1$, $k = n$ and $k=n-1$. Then, assuming the statement holds for $k = l+1$, we prove it must also hold for $k = l$, thereby establishing the result for all values in our domain by backward mathematical induction.

For $k=n+2$, using the boundary condition in Eq. \eqref{eq:main_problem}, we obtain
\begin{align}
\label{eq:k=n+2}
    u(t_{n+2}) &= g(t_{n+2})\notag\\
    &= g(t_{n+2}) + 0\notag \\
    &= g(t_{n+2}) +\displaystyle\sum_{j=n+2}^{n+1}{h_j\displaystyle\sum_{i=1}^{j}{h_{i-1}\widetilde{f}(t_i, u(t_i), u^\Delta(t_{i-1}))}}
\end{align}
Next, for $k = n+1$, using Eq.s \eqref{eq:uDelta} and \eqref{eq:k=n+2}, we obtain
\begin{align} 
\label{eq:k=n+1}
    u^{\Delta}(t_n+1) &= \frac{u(t_{n+2})-u(t_{n+1})}{h_{n+1}}\notag \\
    -\displaystyle\sum_{i=1}^{n+1}{h_{i-1}\widetilde{f}(t_i, u(t_i), u^\Delta(t_{i-1}))} &= \frac{u(t_{n+2})-u(t_{n+1})}{h_{n+1}} \notag\\
    \Rightarrow u(t_{n+1}) &= u(t_{n+2}) + h_{n+1}\displaystyle\sum_{i=1}^{n+1}{h_{i-1}\widetilde{f}(t_i, u(t_i), u^\Delta(t_{i-1}))}\notag \\
    \Rightarrow u(t_{n+1}) &= g(t_{n+2}) +\displaystyle\sum_{j=n+1}^{n+1}{h_j\displaystyle\sum_{i=1}^{j}{h_{i-1}\widetilde{f}(t_i, u(t_i), u^\Delta(t_{i-1}))}}
\end{align}
Then, for $k = n$, using Eq.s \eqref{eq:uDelta} and \eqref{eq:k=n+1}, we obtain
\begin{align}
\label{eq:k=n}
    u^{\Delta}(t_n) &= \frac{u(t_{n+1})-u(t_{n})}{h_{n}}\notag \\
    -\displaystyle\sum_{i=1}^{n}{h_{i-1}\widetilde{f}(t_i, u(t_i), u^\Delta(t_{i-1}))} &= \frac{u(t_{n+1})-u(t_{n})}{h_{n}} \notag \\
    \Rightarrow u(t_{n}) &= u(t_{n+1}) + h_{n}\displaystyle\sum_{i=1}^{n}{h_{i-1}\widetilde{f}(t_i, u(t_i), u^\Delta(t_{i-1}))}\notag \\
    &= g(t_{n+2}) +\displaystyle\sum_{j=n+1}^{n+1}{h_j\displaystyle\sum_{i=1}^{j}{h_{i-1}\widetilde{f}(t_i, u(t_i), u^\Delta(t_{i-1}))}}\notag\\
    &+ h_{n}\displaystyle\sum_{i=1}^{n}{h_{i-1}\widetilde{f}(t_i, u(t_i), u^\Delta(t_{i-1}))}\notag\\
    \Rightarrow u(t_{n}) &= g(t_{n+2}) +\displaystyle\sum_{j=n}^{n+1}{h_j\displaystyle\sum_{i=1}^{j}{h_{i-1}\widetilde{f}(t_i, u(t_i), u^\Delta(t_{i-1}))}}
\end{align}
Thereafter, for $k = n-1$, using Eq.s \eqref{eq:uDelta} and \eqref{eq:k=n}, we obtain
\begin{align}
\label{eq:k=n-1}
    u^{\Delta}(t_{n-1}) &= \frac{u(t_{n})-u(t_{n-1})}{h_{n-1}}\notag \\
    -\displaystyle\sum_{i=1}^{n-1}{h_{i-1}\widetilde{f}(t_i, u(t_i), u^\Delta(t_{i-1}))} &= \frac{u(t_{n})-u(t_{n-1})}{h_{n-1}}\notag\\
    \Rightarrow u(t_{n-1}) &= u(t_{n}) + h_{n-1}\displaystyle\sum_{i=1}^{n-1}{h_{i-1}\widetilde{f}(t_i, u(t_i), u^\Delta(t_{i-1}))} \notag \\
    &= g(t_{n+2}) +\displaystyle\sum_{j=n}^{n+1}{h_j\displaystyle\sum_{i=1}^{j}{h_{i-1}\widetilde{f}(t_i, u(t_i), u^\Delta(t_{i-1}))}} \notag \\
    &+ h_{n-1}\displaystyle\sum_{i=1}^{n-1}{h_{i-1}\widetilde{f}(t_i, u(t_i), u^\Delta(t_{i-1}))} \notag \\
   \Rightarrow u(t_{n-1}) &= g(t_{n+2}) +\displaystyle\sum_{j=n-1}^{n+1}{h_j\displaystyle\sum_{i=1}^{j}{h_{i-1}\widetilde{f}(t_i, u(t_i), u^\Delta(t_{i-1}))}}
\end{align}
Finally, let us write the \textbf{Induction Hypothesis}. Assume the statement (Eq. \eqref{eq:u-k}) is true for $k=l+1$ i.e. 
\begin{equation}
\label{eq:k=l+1}
 u(t_{l+1}) = g(t_{n+2}) +\displaystyle\sum_{j=l+1}^{n+1}{h_j\displaystyle\sum_{i=1}^{j}{h_{i-1}\widetilde{f}(t_i, u(t_i), u^\Delta(t_{i-1}))}}   
\end{equation}
We are to prove, Eq. \eqref{eq:u-k} is true for $k=l$. From Eq.s \eqref{eq:uDelta} and \eqref{eq:k=l+1}, we obtain
\begin{align}
\label{eq:k=l}
    u^{\Delta}(t_l) &= \frac{u(t_{l+1})-u(t_{l})}{h_{l}}\notag \\
    -\displaystyle\sum_{i=1}^{l}{h_{i-1}\widetilde{f}(t_i, u(t_i), u^\Delta(t_{i-1}))} &= \frac{u(t_{l+1})-u(t_{l})}{h_{l}}\notag \\
     \Rightarrow u(t_{l}) &= u(t_{l+1}) + h_{l}\displaystyle\sum_{i=1}^{n}{h_{i-1}\widetilde{f}(t_i, u(t_i), u^\Delta(t_{i-1}))} \notag \\
    &= g(t_{n+2}) +\displaystyle\sum_{j=l+1}^{n+1}{h_j\displaystyle\sum_{i=1}^{j}{h_{i-1}\widetilde{f}(t_i, u(t_i), u^\Delta(t_{i-1}))}}\notag \\
    &+ h_l\displaystyle\sum_{i=1}^{n}{h_{i-1}\widetilde{f}(t_i, u(t_i), u^\Delta(t_{i-1}))}\notag \\
    \Rightarrow u(t_{l}) &= g(t_{n+2}) +\displaystyle\sum_{j=l}^{n+1}{h_j\displaystyle\sum_{i=1}^{j}{h_{i-1}\widetilde{f}(t_i, u(t_i), u^\Delta(t_{i-1}))}}
\end{align}
Consequently, Eq.s \eqref{eq:k=n+2}-\eqref{eq:k=l} successfully establish the lemma.
\end{proof}

\noindent Now we define the solution operator $T: E \to E$ as follows.
\begin{equation}
\label{eq:T}
T(u(t_{k})) := g(t_{n+2}) +\displaystyle\sum_{j=k}^{n+1}{h_j\displaystyle\sum_{i=1}^{j}{h_{i-1}\widetilde{f}(t_i, u(t_i), u^\Delta(t_{i-1}))}}; \quad  k \in \{0, 1, 2, ..., n+2\}
\end{equation}
\noindent Observe that, from Lemma \ref{lem:solution_formula}, if $u$ is a solution of \eqref{eq:aux_problem}, then it is a fixed point of the operator $T$ , i.e. $u(t_k) = T(u(t_k))$ for $k = 0,1, \ldots, n+2$.

\noindent \textbf{Claim}: If $u$ is a fixed point of operator $T$, then $u$ satisfy \eqref{eq:aux_problem} in other words if $T(u(t_k)) = u(t_k)$ then $u(t_k)$ satisfy \eqref{eq:aux_problem}.

\begin{proof}
We will prove our claim in the following three steps. \\   
\noindent \textbf{Step 1:} { \it First we show that $u(t_{n+2}) = g(t_{n+2})$.} Using Eq. \eqref{eq:u-k}, for $k=n+2$ we obtain, 
\begin{align}
\label{eq:revu-n+2}
    u(t_{n+2}) &= g(t_{n+2}) +\displaystyle\sum_{j=n+2}^{n+1}{h_j\displaystyle\sum_{i=1}^{j}{h_{i-1}\widetilde{f}(t_i, u(t_i), u^\Delta(t_{i-1}))}}\notag \\
    \Rightarrow u(t_{n+2}) &= g(t_{n+2}) + 0 \notag \\
    \Rightarrow u(t_{n+2}) &= g(t_{n+2})
\end{align}

\noindent \textbf{Step 2:} {\it Next step we show that $u^{\Delta}(t_0) = 0$}. For our purposes, we first compute $u^\Delta(t_k)$. Recall that, $u^\Delta(t_k) = \frac{u(t_{k+1}) - u(t_k)}{h_k}$. Now using Eq. \eqref{eq:u-k}, we have,

\begin{align}
\label{eq:revu-Delta}
    u^\Delta(t_k) =& \frac{g(t_{n+2}) +\displaystyle\sum_{j=k+1}^{n+1}{h_j\displaystyle\sum_{i=1}^{j}{h_{i-1}\widetilde{f}(t_i, u(t_i), u^\Delta(t_{i-1}))}} - g(t_{n+2}) -\displaystyle\sum_{j=k}^{n+1}{h_j\displaystyle\sum_{i=1}^{j}{h_{i-1}\widetilde{f}(t_i, u(t_i), u^\Delta(t_{i-1}))}}}{h_k} \notag \\
    \Rightarrow u^\Delta(t_k) &= \frac{\displaystyle\sum_{j=k+1}^{n+1}{h_j\displaystyle\sum_{i=1}^{j}{h_{i-1}\widetilde{f}(t_i, u(t_i), u^\Delta(t_{i-1}))}} - \displaystyle\sum_{j=k}^{n+1}{h_j\displaystyle\sum_{i=1}^{j}{h_{i-1}\widetilde{f}(t_i, u(t_i), u^\Delta(t_{i-1}))}}}{h_k}\notag \\
    \Rightarrow u^\Delta(t_k) &= \frac{-h_k\displaystyle\sum_{i=1}^{k}{h_{i-1}}\widetilde{f}(t_i, u(t_i), u^\Delta(t_{i-1}))}{h_k}\notag \\
    \Rightarrow u^\Delta(t_k) &= -\displaystyle\sum_{i=1}^{k}{h_{i-1}}\widetilde{f}(t_i, u(t_i), u^\Delta(t_{i-1}))
\end{align}
Therefore by plugging in $k=0$ in Eq. \eqref{eq:revu-Delta}, we obtain,
\begin{align}
\label{eq:revu-0}
     u^\Delta(t_0) &= -\displaystyle\sum_{i=1}^{0}{h_{i-1}}\widetilde{f}(t_i, u(t_i), u^\Delta(t_{i-1})) \notag \\
    \Rightarrow u^\Delta(t_0) &= 0
\end{align}

\noindent \textbf{Step 3:} {\it Finally, we show that, $u^{\Delta\Delta}(t_{k-1})=-\widetilde{f}(t_k, u(t_k), u^\Delta(t_{k-1}))$ for $k = 1, 2, \ldots, n+1$}. Again recall that for $k = 1, 2, \ldots, n+1$, we have, $u^{\Delta\Delta}(t_k) = \frac{u^\Delta(t_{k+1}) - u^\Delta(t_k)}{h_k}$. 
Now using \eqref{eq:revu-Delta}, we obtain
\begin{align}
\label{eq:revu-DeltaDelta}
    u^{\Delta\Delta}(t_{k-1}) &= \frac{-\displaystyle\sum_{i=1}^{k}{h_{i-1}}\widetilde{f}(t_i, u(t_i), u^\Delta(t_{i-1})) +\displaystyle\sum_{i=1}^{k-1}{h_{i-1}}\widetilde{f}(t_i, u(t_i), u^\Delta(t_{i-1}))}{h_{k-1}}\notag \\
   \Rightarrow u^{\Delta\Delta}(t_{k-1}) &= \frac{-h_{k-1}\widetilde{f}(t_k, u(t_k), u^\Delta(t_{k-1}))}{h_{k-1}} \notag \\
    \Rightarrow u^{\Delta\Delta}(t_{k-1}) &= -\widetilde{f}(t_k, u(t_k), u^\Delta(t_{k-1}))
\end{align}
Consequently, Eq.s \eqref{eq:revu-0}, \eqref{eq:revu-n+2} and \eqref{eq:revu-DeltaDelta} altogether establish the claim.
\end{proof}
Thereupon, one can conclude that $u$ is a solution of \eqref{eq:aux_problem} if and only if $u$ is a fixed point of the operator $T$ given by Eq. \eqref{eq:T}.

\subsection{Fixed Point of Operator $T$}
Having constructed the solution operator $T: E \to E$ in the previous section, we now verify the properties necessary for applying Brouwer's Fixed Point Theorem.

\begin{lem}[Continuity of Operator $T$]
The operator $T: E \to E$ defined by \eqref{eq:T} is continuous.
\end{lem}

\begin{proof}
For each $k \in \{ 0,1, \ldots, n+2\}$, the function $T(u(t_{k-1}))$ is defined as
\begin{equation}
T(u(t_{k})) = g(t_{n+2}) + \sum_{j=k}^{n+1} h_j \sum_{i=1}^j h_{i-1} \widetilde{f}(t_i, u(t_i), u^\Delta(t_{i-1})).
\end{equation}

Since auxiliary function $\widetilde{f}$ is continuous, $g$ is continuous by assumption {\bf (A3)}, the operator $T$ is continuous as a finite composition of continuous functions.
\end{proof}

\begin{pro}[Invariance of Ball]
Let $\overline{B_r} = \{u \in E : \|u\|_\infty \leq r\}$ where $\|u\|_\infty = \max_{t \in \mathbb{T}} |u(t)|$. If 
\begin{equation}
r \geq g(t_{n+2}) + M(t_{n+2} - t_0)^2,
\end{equation}
where $M$ is the upperbound of $\widetilde{f}$,
then $T(\overline{B_r}) \subseteq \overline{B_r}$.
\end{pro}

\begin{proof}
Let $u \in \overline{B_r}$, so $\|u\|_\infty \leq r$. We must show that $\|Tu\|_\infty \leq r$.

For any $k \in \{0, 1, \ldots, n+2\}$, we have
\begin{align*}
|T(u(t_k))| &= \left|g(t_{n+2}) + \sum_{j=k}^{n+1} h_j \sum_{i=1}^j h_{i-1} \widetilde{f}(t_i, u(t_i), u^\Delta(t_{i-1}))\right| \\
&\leq |g(t_{n+2})| + \left|\sum_{j=k}^{n+1} h_j \sum_{i=1}^j h_{i-1} \widetilde{f}(t_i, u(t_i), u^\Delta(t_{i-1}))\right| \\
&\leq g(t_{n+2}) + \sum_{j=k}^{n+1} h_j \sum_{i=1}^j h_{i-1} |\widetilde{f}(t_i, u(t_i), u^\Delta(t_{i-1}))| \\
&\leq g(t_{n+2}) + M \sum_{j=k}^{n+1} h_j \sum_{i=1}^j h_{i-1},
\end{align*}
where the last inequality follows from the boundedness of $\widetilde{f}$. 
Since $j \leq n+1$,
\begin{align*}
\sum_{i=1}^{j} h_{i-1} &\leq \sum_{i=1}^{n+1} h_{i-1} \\
&= \sum_{i=1}^{n+1} (t_i - t_{i-1}) \\
&= t_1 - t_0 + t_2 - t_1 + \cdots + t_{n+1} - t_n \\
&= (t_{n+1} - t_0) < (t_{n+2} - t_0),
\end{align*}
and since $k \geq 0$,
\begin{align*}
\sum_{j=k}^{n+1} h_j &\leq \sum_{j=0}^{n+1} h_j \\
&= \sum_{j=0}^{n+1} (t_{j+1} - t_j) \\
&= t_1 - t_0 + t_2 - t_1 + \cdots + t_{n+2} - t_{n+1} \\
&= (t_{n+2} - t_0).
\end{align*}
Therefore,
$|T(u(t_k))| < g(t_{n+2}) + M(t_{n+2} - t_0)^2$.
By our choice of $r$, this gives $|T(u(t_{k}))| \leq r$ for all $k \in \{0, 1, \ldots, n+2\}$. Consequently, $\|Tu\|_\infty = \max_{k} |T(u(t_{k}))| \leq r$, which implies $Tu \in \overline{B_r}$.
\end{proof}
Therefore, employing Brouwer's fixed point theorem (see Corollary \ref{cor:Brouwer}) one can conclude that $T$ has a fixed point in $\overline{B(r)}$. In other words, there exists $u \in \overline{B(r)}$, such that $T(u^*(t_k)) = u^*(t_k)$ for all $k \in \{0,1, \ldots, n+2\}$.
\begin{rem}
    {\rm Note that in the two dimensional setting, we can think of $\overline{B(r)}$ as a circle of radius $r$, whereas in three dimensional setting, as a sphere with radius $r$.}
\end{rem}

\section{Verification of Bounds}
\label{sec:bound}

The final step shows any solution $u$ satisfies $\alpha(t_i) \leq u(t_i) \leq \beta(t_i)$ for all $i \in \{0, 1, \ldots, n+2\}$. We prove this by contradiction in two parts.

\subsection{Proof that $u(t_i) \leq \beta(t_i)$}

Suppose there exists some index $\ell \in \{0, 1, \ldots, n+2\}$ such that $u(t_\ell) > \beta(t_\ell)$. Define the maximum violation
\begin{equation}
M = \max\{u(t_i) - \beta(t_i) : i \in \{0, 1, \ldots, n+2\}\} > 0.
\end{equation}
Let $\ell$ be an index where this maximum is achieved, so $u(t_\ell) - \beta(t_\ell) = M > 0$.

We consider three cases for the location of $\ell$.

\textbf{Case 1:} $\ell = 0$. From the boundary condition $u^\Delta(t_0) = 0$ and the upper solution condition $\beta^\Delta(t_0) \leq 0$, we have
\begin{equation*}
u^\Delta(t_0) - \beta^\Delta(t_0) = 0 - \beta^\Delta(t_0) \geq 0.
\end{equation*}
This implies
\begin{equation*}
\frac{u(t_1) - u(t_0)}{h_0} \geq \frac{\beta(t_1) - \beta(t_0)}{h_0},
\end{equation*}
which gives us $u(t_1) - \beta(t_1) \geq u(t_0) - \beta(t_0) = M$. Since $M$ is the maximum violation, we must have $u(t_1) - \beta(t_1) = M$, meaning the maximum is also achieved at $t_1$. We can therefore consider $\ell = 1$ instead.

\textbf{Case 2:} $\ell = n+2$. From the boundary conditions $u(t_{n+2}) = g(t_{n+2})$ and $\beta(t_{n+2}) \geq g(t_{n+2})$, we obtain
\begin{equation*}
u(t_{n+2}) - \beta(t_{n+2}) = g(t_{n+2}) - \beta(t_{n+2}) \leq 0.
\end{equation*}
This contradicts our assumption that $u(t_{n+2}) - \beta(t_{n+2}) = M > 0$.

\textbf{Case 3:} $\ell \in \{1, 2, \ldots, n+1\}$. Since $\ell$ achieves the maximum of $u(t_i) - \beta(t_i)$, we have $u(t_i) - \beta(t_i) \leq u(t_\ell) - \beta(t_\ell) = M$ for all $i$. Since $\ell$ is a maximum point of the function $w(t_i) = u(t_i) - \beta(t_i)$, the discrete second derivative satisfies
\begin{equation}
\label{eq:max-cond}
(u - \beta)^{\Delta\Delta}(t_{\ell-1}) \leq 0.
\end{equation}
This follows from the discrete analogue of the second derivative test.

From the auxiliary problem \eqref{eq:aux_problem}, $u$ satisfies
\begin{equation*}
u^{\Delta\Delta}(t_{\ell-1}) + \widetilde{f}(t_\ell, u(t_\ell), u^\Delta(t_{\ell-1})) = 0.
\end{equation*}
and from the definition of upper solution outlined in \eqref{def:upper_solution}, we have
\begin{equation*}
\beta^{\Delta\Delta}(t_{\ell-1}) + f(t_\ell, \beta(t_\ell), \beta^\Delta(t_{\ell-1})) \leq 0.
\end{equation*}

Since $u(t_\ell) > \beta(t_\ell)$, by plugging in $x=u(t_{\ell})$ and $z=u(t_{\ell-1})$ in the definition of $\widetilde{f}$ outlined in equation \eqref{eq:tilde-f}, we have 
\begin{align}
\label{eq:tildef-beta}
\widetilde{f}(t_\ell, u(t_\ell), u^\Delta(t_{\ell-1})) = & \widetilde{f}\left(t_\ell, u(t_\ell),\frac{u(t_{\ell})-u(t_{\ell-1})}{h_{\ell -1}}\right) \notag \\
\Rightarrow \widetilde{f}(t_\ell, u(t_\ell), u^\Delta(t_{\ell-1}))= & f\left(t_{\ell}, \beta(t_{\ell}), \frac{\beta(t_{\ell})-\sigma(t_{\ell}, u(t_{\ell-1}))}{h_{\ell-1}}\right) - \frac{u(t_{\ell})-\beta(t_{\ell})}{u(t_{\ell})-\beta(t_{\ell}) + 1} \notag \\
\Rightarrow -\widetilde{f}(t_\ell, u(t_\ell), u^\Delta(t_{\ell-1}))= & -f\left(t_{\ell}, \beta(t_{\ell}), \frac{\beta(t_{\ell})-\sigma(t_{\ell}, u(t_{\ell-1}))}{h_{\ell-1}}\right) + \frac{u(t_{\ell})-\beta(t_{\ell})}{u(t_{\ell})-\beta(t_{\ell}) + 1}.
\end{align}
Now, from the definition of $\sigma$ given in Eq. \eqref{eq:sigma},
\begin{equation*}
\sigma(t_{\ell}, u(t_{\ell-1})) = \begin{cases}
 \beta(t_{\ell-1}), & \text{if } u(t_{\ell-1}) > \beta(t_{\ell-1}),\\
u(t_{\ell-1}), & \text{if } \alpha(t_{\ell-1}) \leq u(t_{\ell-1}) \leq \beta(t_{\ell-1}), \\
\alpha(t_{\ell-1}), & \text{if } u(t_{\ell-1}) < \alpha(t_{\ell-1}),
\end{cases}
\end{equation*}
which implies $\sigma(t_{\ell}, u(t_{\ell-1})) \leq \beta(t_{\ell-1})$. Therefore, using the assumption {\bf(A2)}: $f$ is non-increasing in its third variable we obtain, 
\begin{equation}
\label{eq:f-beta}
    -f\left(t_{\ell}, \beta(t_{\ell}), \frac{\beta(t_{\ell})-\sigma(t_{\ell}, u(t_{\ell-1}))}{h_{\ell-1}}\right) \geq - f(t_\ell, \beta(t_\ell), \beta^\Delta(t_{\ell-1}))
\end{equation}
Therefore, combining Eq.s \eqref{eq:aux_problem}, \eqref{def:upper_solution}, \eqref{eq:tildef-beta} and \eqref{eq:f-beta}, we obtain, 
\begin{equation}
\label{eq:u-contr}
u^{\Delta\Delta}(t_{\ell-1}) - \beta^{\Delta\Delta}(t_{\ell-1}) \geq \frac{u(t_\ell) - \beta(t_\ell)}{u(t_\ell) - \beta(t_\ell) + 1} > 0.    
\end{equation}
which clearly contradicts Eq. \eqref{eq:max-cond}. This completes the proof.

\subsection{Proof that $\alpha(t_i) \leq u(t_i)$} Proof of this part will follow the similar argument presented above by considering the set $\{\alpha(t_i) - u(t_i) : i \in \{0, 1, \ldots, n+2\}\}$.

Consequently, we have, $\alpha(t_i) \leq u(t_i) \leq \beta(t_i)$ for all $i \in \{0, 1, \ldots, n+2\}$.

\end{document}